\documentclass[11pt]{amsart}
\usepackage[margin=2.5cm]{geometry}
\usepackage{mathrsfs}
\usepackage{epic,eepic,epsf,epsfig,tikz}
\usetikzlibrary{positioning}
\usepackage{amsmath}
\usepackage{amssymb}
\usepackage{amsbsy}
\usepackage{amsfonts}
\usepackage{color}

\numberwithin{equation}{section}
\newtheorem{lem}{Lemma}[section]%
\newtheorem{theorem}[lem]{Theorem}%
\newtheorem{prop}[lem]{Proposition}%

\def\a{\alpha}

  \def\G{\Gamma}

\def\nd{\mathrel{\bigm|\kern-.7em/}}

\def\f{\noindent}

\def\Aut{\hbox{\rm Aut}}
\def\Cay{\hbox{\rm Cay}}

\def\mz{{\mathbb Z}}

\begin{document}
\title[]{On $n$-partite digraphical representations of finite groups}

\author{Jia-Li Du}
\address{Jia-Li Du, School of Mathematics, China University of Mining and Technology, Xuzhou 221116, China}
\email{dujl@cumt.edu.cn}

\author{Yan-Quan Feng*}
\address{Yan-Quan Feng, Department of Mathematics, Beijing Jiaotong University, Beijing 100044, China}
\email{yqfeng@bjtu.edu.cn}

\author{Pablo Spiga*}
\address{Pablo Spiga, Dipartimento di Matematica e Applicazioni, University of Milano-Bicocca, Via Cozzi 55, 20125 Milano, Italy}
\email{pablo.spiga@unimib.it}

\footnotetext[1]{Corresponding authors}
\date{}
 \maketitle

\begin{abstract}

A group $G$ admits an \textbf{\em $n$-partite digraphical representation} if there exists a regular $n$-partite digraph $\Gamma$ such that the automorphism group $\mathrm{Aut}(\Gamma)$ of $\Gamma$ satisfies the following properties:
 \begin{itemize}
 \item [\rm(1)]$\mathrm{Aut}(\Gamma)$ is isomorphic to $G$,
 \item [\rm(2)]$\mathrm{Aut}(\Gamma)$ acts semiregularly on the vertices of $\Gamma$ and
 \item [\rm(3)]the orbits of $\mathrm{Aut}(\Gamma)$ on the vertex set of $\Gamma$ form a partition into $n$ parts giving a structure of $n$-partite digraph to $\Gamma$.
 \end{itemize}

  In this paper, for every positive integer $n$, we classify the finite groups admitting an $n$-partite digraphical representation.

\bigskip
\f {\bf Keywords:} Semiregular group, regular representation, DRR, D$n$SR, $n$-PDR.

\medskip
\f {\bf 2010 Mathematics Subject Classification:} 05C25, 20B25.
\end{abstract}

\section{Introduction}

We start by summarising some basic definitions. A \textbf{\em digraph} $\Gamma$ is an ordered pair $(V\Gamma, A\Gamma)$, where $V\Gamma$ is a non-empty set and $A\Gamma$ is a subset of  $V\Gamma \times V\Gamma$. It is customary to refer to $V\Gamma$ and to $A\Gamma$ as the \textbf{{\em vertex set}} and the \textbf{{\em arc set}} of $\Gamma$, respectively, and to refer to their elements as \textbf{\em vertices} and \textbf{\em arcs}. Given an arc $(u,v)\in A\Gamma$, $v$ is an \textbf{{\em out-neighbour}} of $u$ and $u$ is an \textbf{{\em in-neighbour}} of $v$. The \textbf{{\em out-valency}} and the \textbf{{\em in-valency}} of $v\in V\Gamma$ is the number of out-neighbours and the number of in-neighbours of $v$. Moreover, $\Gamma$ is \textbf{\em regular} if there exists a non-negative integer $d$ such that every vertex $v\in V\Gamma$ has out-valency and in-valency $d$.
Given a subset $X$ of $V\Gamma$, we denote by $\Gamma[X]:=(V\Gamma\cap X,A\Gamma\cap (X\times X))$ the \textbf{\em sub-digraph induced} by $\Gamma$ on $X$.

The digraph $\Gamma$  is a \textbf{\em graph} if $A\Gamma$ is symmetric, that is, $A\Gamma=\{(u,v)\ |\ (v,u)\in A\Gamma \}$.

An \textbf{{\em automorphism}} of $\Gamma$ is
a permutation $\sigma$ of $V\Gamma$ such that, for every $(u,v)\in V\Gamma\times V\Gamma$, $(u^\sigma , v^\sigma ) \in A\Gamma$ if and only if $(u, v) \in A\Gamma$. The \textbf{{\em automorphism group}} $\mathrm{Aut}(\Gamma)$ of $\Gamma$ is the set of all automorphisms of $\Gamma$ and it  is indeed a group under composition of permutations.

Let  $G$ be a permutation group on a set $\Omega$ and let $\omega \in \Omega$. We let $G_\omega$ denote the \textbf{{\em stabiliser}} of $\omega$ in $G$, that is, the subgroup
of $G$ fixing $\omega$. We say that $G$ is \textbf{\em semiregular} if $G_\omega = 1$ for every $\omega \in \Omega$, and \textbf{\em regular} if it is semiregular and transitive. We denote by $\mathbb{Z}_k$ the cyclic group of order $k\in\mathbb{N}\setminus\{0\}$ and by $Q_8$ the quaternion group of order $8$.

We are now ready to give the definition of the main player in this paper. Let $G$ be a group and let $n$ be a positive integer. An \textbf{\em $n$-partite digraphical representation}
($n$-$\rm PDR$ for short) of $G$ is a digraph $\Gamma=(V\Gamma,A\Gamma)$ such that
 \begin{itemize}
 \item $\Gamma$ is regular,
 \item $\mathrm{Aut}(\Gamma)$ is isomorphic to $G$,
 \item $\mathrm{Aut}(\Gamma)$ acts semiregularly on $V\Gamma$,
 \item  $\mathrm{Aut}(\Gamma)$ has $n$ orbits on $V\Gamma$ and, for every orbit $X$, the sub-digraph $\Gamma[X]$ induced by $\Gamma$ on $X$ is the empty graph.
 \end{itemize}
Observe that the last condition implies that $\Gamma$ is $n$-partite. The scope of this paper is to classify finite groups admitting $n$-PDRs.
\begin{theorem}\label{theo=main}
Let $G$ be a finite group and let $n$ be a positive integer. Then $G$ admits no $n$-partite digraphical representation if and only if one of the following occurs:
\begin{enumerate}
\item\label{part1} $n=1$ and $|G|\geq 3$,
\item\label{part2} $n=2$ and $G\cong \mz_1$, $\mz_2$, $\mz_3$, $\mz_2\times \mz_2$ or $\mz_2\times \mz_2\times \mz_2$;
\item\label{part3} $3\leq n\leq 5$ and $G\cong \mz_1$.
\end{enumerate}
\end{theorem}

We now discuss the roots of our interest on Theorem~\ref{theo=main}.
The study of group representations on digraphs starts with the classical work on DRRs. A \textbf{\em digraphical regular representation} of $G$ is a digraph $\Gamma=(V\Gamma,A\Gamma)$ such that $\Aut(\G)\cong G$ and $\Aut(\G)$ is regular on $V\G$.  Babai~\cite{Babai} has proved that, except for
$$
Q_8,\,
\mz_2\times \mz_2,\,
\mz_2\times \mz_2\times \mz_2,\,
\mz_2\times \mz_2\times \mz_2\times \mz_2\,\hbox{ and }
\mz_3\times \mz_3,$$ every finite group admits a DRR. This classification plays a crucial role in our work. Observe that the case $n=1$ in Theorem~\ref{theo=main} is trivial. Indeed, the last condition in the definition of $n$-PDR implies and a $1$-PDR is an empty graph. Therefore, a group having order at least $3$ does not admit a $1$-PDR, whereas $\mz_1$ and $\mz_2$ do admit a $1$-PDR.

After the classification of finite groups admitting a DRR was completed, researchers proposed and investigated various natural generalisations. For instance, Babai and Imrich~\cite{BabaiI} classified finite groups admitting a tournament regular representation. Morris and Spiga~\cite{MorrisSpiga1,MorrisSpiga3,Spiga}, answering a question of Babai~\cite{Babai}, classified the finite groups admitting an oriented regular representation.
For more results, generalising DRRs in various directions, we refer to
\cite{DSV,DFS,DFS1,MSV,Spiga2,Xiaf}.
To give some more context to Theorem~\ref{theo=main}, we give  details to two more particular generalisations.

Given a positive integer $n$, a \textbf{{\em digraphical $n$-semiregular representation}} ({\em \rm D$n$SR} for short) of $G$ is a regular digraph $\G=(V\Gamma,A\Gamma)$ such that $\Aut(\G)\cong G$ is semiregular on $V\G$ with $n$ orbits. Observe that every $n$-PDR is also a D$n$SR, but not every D$n$SR is necessarily a $n$-PDR because it may not be $n$-partite. Therefore, Theorem~\ref{theo=main} can be seen as an extension of the classification in~\cite{DFS} of the finite groups admitting a D$n$SR. The second generalisation we discuss is concerned with Haar digraphs. Given a group $G$, a \textbf{{\em Haar digraph}} $\Gamma$ over $G$ is a bipartite digraph having a bipartition $\{X,Y\}$ such that $G$ is a group of automorphisms of $\Gamma$ acting regularly on $X$ and on $Y$. We say that $\Gamma$ is a \textbf{\em Haar digraphical representation} (HDR for short) of $G$, if there exists a regular Haar digraph over $G$ such that its automorphism group is isomorphic to $G$. We have proved in~\cite{DFS2} that, except for
$$\mz_1,\, \mz_2,\, \mz_3,\, \mz_2\times \mz_2\, \hbox{ and }\mz_2\times \mz_2\times \mz_2,$$
every finite group admits a HDR.
Observe that a digraph $\Gamma$ is a  HDR of $G$ if and only if $\Gamma$ is a $2$-PDR of $G$. Therefore, $n$-PDRs offer also a natural generalisation of HDRs. Furthermore, the classification in~\cite{DFS2} of finite groups not admitting HDRs gives a classification of finite groups not admitting $2$-PDRs. In particular, Theorem~\ref{theo=main}~\eqref{part2} follows from~\cite{DFS2}. Thus, in the proof of Theorem~\ref{theo=main}, we may suppose that
\begin{align}\label{boundn}n\ge 3.
\end{align}

 Despite the natural argument used by Babai for the classification of finite groups admitting a DRR, the classification of finite groups admitting a GRR has required considerable more work. Indeed, as a rule-of-thumb, representation results dealing with undirected graphs are difficult. For instance, 
 we classify the finite groups $G$ having an abelian subgroup $A$ of index $2$ admitting a bipartite DRR having bipartition $\{A,G\setminus A\}$. The analogous classification for Cayley graphs is much harder, see~\cite{DFS1,DFS3}.
  In fact, we do not have a classification of finite groups admitting a Haar graphical representation and hence we have no classification of finite groups admitting an $n$-partite graphical representation, when $n\ge2$. Incidentally, the case $n=2$ is in our opinion the most important. In this context, we believe that the classification of finite groups admitting a bipartite regular representation can be of some relevance, see the introductory section in~\cite{DFS1} for more details.

We conclude this section observing that the work of Grech and Kisielewicz~\cite{G1,G2,G3}, which in spirit is trying to give an explicit description of some classes of permutation groups that are $2$-closed, does fit within the subject of group representations on digraphs. In particular, some of the results of Grech and Kisielewicz are interesting in the context of DRRs and their generalisations.
\section{Preliminaries and notations}

Let $n$ be a positive integer and let $G$ be a finite group. Consistently throughout the whole paper, for not making our notation too cumbersome to use, we denote the element $(g,i)$ of the cartesian product $G\times\{0,\ldots,n-1\}$ simply by $g_i$. We often identify $\{0,\ldots,n-1\}$ with $\mathbb{Z}_n$, that is, with the integers modulo $n$.

For every $i,j\in\mathbb{Z}_n$, let $T_{i,j}$ be a subset of $G$. The \textbf{{\em $n$-Cayley digraph}} of $G$ with respect to $(T_{i,j}: i,j\in\mz_n)$ is the digraph with vertex set $$G\times \mathbb{Z}_n=\bigcup_{i\in\mz_n} G_i,$$ where $G_i=\{g_i\ |\ g\in G\}$, and with arc set  $$\bigcup_{i,j} \{(g_i, (tg)_j)~|~t\in T_{i,j},g\in G\}.$$
We denote this digraph by  $$\Cay(G,T_{i,j}: i,j\in\mz_n).$$ Observe that, when $n:=1$, $1$-Cayley digraphs are nothing more and nothing less than Cayley digraphs. Similarly, when $n:=2$, $2$-Cayley digraphs are also known as BiCayley digraphs in the literature. When dealing with Cayley digraphs, we omit the subscript, that is, we denote the Cayley digraph of $G$ with connection set $S$ with $\Cay(G,S)$.

Set $\Gamma:=\Cay(G,T_{i,j}: i,j\in\mz_n)$. For every $g\in G$, the mapping
\begin{align*}
R_n(g):\,& V\Gamma\to V\Gamma\\
&x_i\mapsto (xg)_i
\end{align*}
 is an automorphism  of $\Gamma$. Thus $\{R_n(g)\ |\ g\in G\}$ is a subgroup of $\Aut(\Gamma)$ isomorphic to $G$. Since we are not interested in $G$ as an abstract group, but rather as a group of automorphisms of digraphs, for convenience, we identify $\{R_n(g)\ |\ g\in G\}$ with $G$. Observe that
\begin{itemize}
\item $G$ acts semiregularly on $V\Gamma=G\times \mathbb{Z}_n$,
\item $G$ has $n$ orbits on $V\Gamma$ and  these are $G_i$, with  $i\in\mathbb{Z}_n$.
\end{itemize}

 It is easy to see that a digraph $\Delta$ is an $n$-Cayley digraph of $G$ if and only if $\mathrm{Aut}(\Delta)$ contains a semiregular subgroup isomorphic to $G$ and with $n$ orbits on $V\Delta$. (When $n:=1$, this is a classic observation of Sabidussi~\cite[Lemma~4]{Sabidussi}; for the proof of the general case, it suffices to follow the argument of Sabidussi as a crib.)

 In our work we need the following two results concerning DRRs: Proposition~\ref{prop=DRR} is the classification of Babai~\cite[Theorem 2.1]{Babai} of finite groups admitting a DRR which we mentioned in the introduction, Proposition~\ref{prop=CayleySet} is a technical result proved in~\cite[Lemma 3.5]{DFS2}.

\begin{prop}\label{prop=DRR}
A finite group $G$ admits a $\mathrm{DRR}$ if and only if $G$ is not isomorphic to one of the following five groups: $Q_8$, $\mz_2^2$, $\mz_2^3$, $\mz_2^4$ and $\mz_3^2$.
\end{prop}

\begin{prop}\label{prop=CayleySet}
Let $G$ be a finite group of order at least $4$ admitting a $\mathrm{DRR}$.
Then $G$ has a subset $R$ such that $\Cay(G,R)$ is a $\mathrm{DRR}$, where $1\notin R$ and $|R|<(|G|-1)/2$.
\end{prop}

Let $\Gamma:=\Cay(G,T_{i,j}: i,j\in\mz_n)$ be an $n$-Cayley digraph of $G$ with
$$T_{i,i}=\emptyset,\,\,\forall i\in\mathbb{Z}_n.$$
Then, for every $i\in \mathbb{Z}_n$, $\Gamma[G_i]$ is the empty graph. Thus $\Gamma$ is  an $n$-partite digraph where  every part is an orbit of $G$. We call these digraphs  \textbf{{\em $n$-partite Cayley digraphs}}.
Summing up, a group $G$ admits an $n$-$\rm PDR$ if and only if $G$ admits a regular $n$-partite Cayley digraph having $G$ as its automorphism group.

As we mentioned in the introduction, a regular $2$-partite Cayley digraph of $G$ is also called Haar digraph of $G$ (see~\cite{DFS2}). Moreover, since $T_{0,0}=T_{1,1}=\emptyset$, it can be written as  $\Cay (G,T_{0,1},T_{1,0})$.

In our work we need the following two results concerning $2$-PDRs (a.k.a. HDRs, for Haar digraphical representations): Proposition~\ref{prop=HDR} is the classification of finite groups admitting a HDR~\cite[Theorem~1.1]{DFS2}, Proposition~\ref{prop=HDRSet} is a technical result whose proof follows from the proof of  Lemma 3.4 and Theorem 1.1 in~\cite{DFS2}.

\begin{prop}\label{prop=HDR}
A finite group $G$ admits a $2$-$\mathrm{PDR}$ if and only if $G$ is not isomorphic to one of the following five groups: $\mz_1$, $\mz_2$, $\mz_3$, $\mz_2^2$ and $\mz_2^3$.
\end{prop}

\begin{prop}\label{prop=HDRSet}
Let $G$ be a finite group and let $\Cay(G,R)$ be a $\mathrm{DRR}$ of $G$, where $1\notin R\subset G$ and $|R|<|G|/2$.
Then $G$ has a subset $L$ such that $\Cay(G,R\cup\{1\},L\cup\{1\})$ is a $2$-$\mathrm{PDR}$, where
$L\subseteq G\setminus (R^{-1}\cup\{1\})$ and $|L|=|R|$.
\end{prop}

We conclude this section with some results concerning small troublesome groups.

\begin{lem}\label{lem=small}The following hold:
\begin{enumerate}
\item\label{lem=small1}$\mz_1$ admits an $n$-$\mathrm{PDR}$ if and only if $n=1$ or $n\ge 6$,
\item\label{lem=small2}$\mz_1$ admits an $n$-$\mathrm{PDR}$ if and only if $n=1$ or $n\ge 3$,
\item\label{lem=small3}$\mz_3$ admit an $n$-$\mathrm{PDR}$ if and only if $n\ge 3$.
\end{enumerate}
\end{lem}

\begin{proof} We start by dealing with $\mathbb{Z}_1$. Observe that over $\mathbb{Z}_1$, $n$-PDRs and D$n$SRs define the same family of digraphs. Now, by~\cite[Theorem 1.2(4)]{DFS},  $\mz_1$ admits a D$n$SR if and  only if $n=1$ or $n\geq 6$. Thus part~\eqref{lem=small1} immediately follows.

From~\eqref{boundn}, $n\geq 3$. Let $G$ be either $\mz_2$ or $\mz_3$, let $a$ be a generator of $G$ and let
\begin{eqnarray*}
T_{i,i+1}&=&T_{i+1,i}:=\{1\},\mbox{ when } i\in\mz_n\setminus\{1\}, \\
T_{2,1}&=&T_{1,2}:=\{a\},\\
 T_{j,k}&:=&\emptyset,\qquad\hbox{ otherwise}.
\end{eqnarray*}
Let $\Gamma:=\Cay(G,T_{i,j}: i,j\in\mz_n)$ and let $A:=\Aut(\G)$. See Figure~\ref{Fig1} for a rough drawing of $\Gamma$. Clearly, $\Gamma$ is a regular $n$-partite Cayley digraph for $G$. Therefore, to finish the proof, we only need to show that $A=G$. To do that, we use the fact that,
\begin{center}
$(\dag)\qquad$there is a unique undirected path from $1_1$ to $1_2$ of length $n-1$ (namely, the path $1_1,1_0,1_{n-1},1_{n-2}, \ldots, 1_3,1_2$), whereas, there is no such path from $1_2$ to $a_1$.
\end{center}
\begin{figure}[!ht]
\begin{tikzpicture}[node distance=1.2cm,thick,scale=0.7,every node/.style={transform shape},scale=1.8]
\node[circle](A00){};
\node[left=of A00](A000){};
\node[left=of A000](A0001){};
\node[left=of A0001, circle,draw,inner sep=1pt, label=above:{$1_0$}](A0){};
\node[right=of A0, circle,draw, inner sep=1pt, label=above:{$1_1$}](B0){};
\node[below=of A0, circle,draw, inner sep=1pt, label=below:{$a_0$}](A1){};
\node[right=of A1, circle,draw, inner sep=1pt, label=below:{$a_1$}](B1){};
\node[right=of B0, circle,draw, inner sep=1pt, label=above:{$1_2$}](C0){};
\node[below=of C0, circle,draw, inner sep=1pt, label=below:{$a_2$}](C1){};
\node[right=of C0, circle,draw, inner sep=1pt, label=below:{$1_{n-1}$}](B11){};
\node[below=of B11, circle,draw, inner sep=1pt, label=below:{$a_{n-1}$}](B12){};

\draw (A0) to  (B0);
\draw (A1) to (B1);
\draw[->] (B0) to (C0);
\draw[->] (B1) to (C1);
\draw[->] (C0) to (B1);
\draw[->] (C1) to (B0);
\draw[dashed] (C0) to (B11);
\draw[dashed] (C1) to (B12);
\draw (A0) to [bend left] node [above]{} (B11);
\draw (A1) to [bend right] node [above]{}(B12);

\node[right=of C0](D00){};
\node[right=of D00, circle,draw, inner sep=1pt, label=above:{$1_0$}](D0){};
\node[right=of D0, circle,draw, inner sep=1pt, label=above:{$1_1$}](E0){};
\node[below=of D0, circle,draw, inner sep=1pt, label=below:{$a_0$}](D1){};
\node[below=of D1, circle,draw, inner sep=1pt, label=below:{$(a^2)_0$}](D2){};
\node[right=of E0, circle,draw, inner sep=1pt, label=above:{$1_2$}](F0){};
\node[below=of E0, circle,draw, inner sep=1pt, label=below:{$a_1$}](E1){};
\node[below=of E1, circle,draw, inner sep=1pt, label=below:{$(a^2)_1$}](E2){};
\node[below=of F0, circle,draw, inner sep=1pt, label=below:{$a_2$}](F1){};
\node[below=of F1, circle,draw, inner sep=1pt, label=below:{$(a^2)_2$}](F2){};
\node[right=of F0, circle,draw, inner sep=1pt, label=below:{$1_{n-1}$}](G0){};
\node[below=of G0, circle,draw, inner sep=1pt, label=below:{$a_{n-1}$}](G1){};
\node[below=of G1, circle,draw, inner sep=1pt, label=below:{$(a^2)_{n-1}$}](G2){};

\draw (D0) to (E0);
\draw (D1) to (E1);
\draw (D2) to (E2);
\draw[->] (E0) to (F0);
\draw[->] (E1) to (F1);
\draw[->] (E2) to (F2);
\draw[->] (F0) to (E1);
\draw[->] (F1) to (E2);
\draw[->] (F2) to (E0);
\draw[dashed] (F0) to (G0);
\draw[dashed] (F1) to (G1);
\draw[dashed] (F2) to (G2);
\draw (D0) to [bend left] node [above]{} (G0);
\draw (D1) to [bend left] node [above]{}(G1);
\draw (D2) to [bend left] node [above]{}(G2);
\end{tikzpicture}
\caption{$n$-$\rm PDR$s for $\mathbb{Z}_2$ and $\mz_3$ with $n\geq 3$}\label{Fig1}
\end{figure}

Observe now that
\begin{itemize}
\item $\G[G_1\cup G_2]$ is a directed cycle  of length four when $G\cong \mz_2$ and of length six when $G\cong \mz_3$ (every arc on the cycle is a directed edge),
\item for $i\in\mz_n\setminus\{1\}$, $\G[G_i\cup G_{i+1}]$ is a
perfect matching,
\item in all other cases, $\G[G_j \cup G_k]$ is the empty graph.
\end{itemize}See again Figure~\ref{Fig1}. Thus, $A$ fixes $G_1\cup G_2$ setwise. In particular, $A$ induces a group of automorphisms of the directed cycle $\Gamma[G_1\cup G_2]$. Therefore, $\{G_1,G_2\}$ is a system of imprimitivity for the (not necessarily transitive) action of $A$ on $G_1\cup G_2$.

Suppose $A$ does not fix setwise $G_1$. As $G\le A$ and as $G$ acts transitively on $G_1$ and on $G_2$, there exists $\a\in A$ with $1_1^{\a}=1_2$. Since $(1_1,1_2)$ is an arc of $\Gamma[G_1\cup G_2]$, $(1_1,1_2)^\a=(1_1^\a,1_2^\a)=(1_2,1_2^\a)$ is also an arc of $\Gamma[G_1\cup G_2]$ and hence $1_2^\a=a_1$. However, this contradicts~$(\dag)$. Thus $A$ fixes $G_1$ and $G_2$ setwise. Since $G$ acts transitively on $G_2$, from the Frattini argument, we have  $A=GA_{1_2}$. Now, as $\Gamma[G_1\cup G_2]$ is a directed cycle, it follows easily that $A_{1_2}$ fixes $G_1$ and $G_2$ pointwise.

Since $\G[G_2\cup G_3]$ is a perfect matching, $A_{1_2}$ fixes $G_3$ pointwise, and arguing inductively, $A_{1_2}$ fixes
$G_i$ pointwise for every $i\in\mz_n$. It follows $A_{1_2}=1$ and $A=GA_{1_2}=G$.
\end{proof}

\begin{lem}\label{lem=2^2,2^3}
The groups $\mz_2^2$ and $\mz_2^3$ admit an $n$-$\mathrm{PDR}$ if and only $n\ge 3$.
\end{lem}
\begin{proof}
From~\eqref{boundn}, $n\geq 3$. Suppose first $G\cong\mz_2^2$.  Let $a,b\in G$ with  $G=\langle a\rangle\times \langle b\rangle$ and let
\begin{eqnarray*}
T_{i,i+1}&=&T_{i+1,i}:=\{1\},\,\,\mbox{ when } i\in\mz_n\setminus\{0,1\}, \\
T_{0,1}&:=&\{1\},\,\,T_{1,0}:=\{a\},\,\,T_{1,2}:=\{b\},\,\,T_{2,1}:=\{a\},\\
T_{j,k}&:=&\emptyset,\,\,\hbox{ otherwise}.
\end{eqnarray*}
Let $\Gamma:=\Cay(G,T_{i,j}: i,j\in\mz_n)$ and let $A:=\Aut(\G)$. See Figure~\ref{Fig2} for a rough drawing of $\Gamma$. Clearly, $\Gamma$ is a regular $n$-partite Cayley digraph of $G$. To finish the proof, we only need to show that $A=G$.

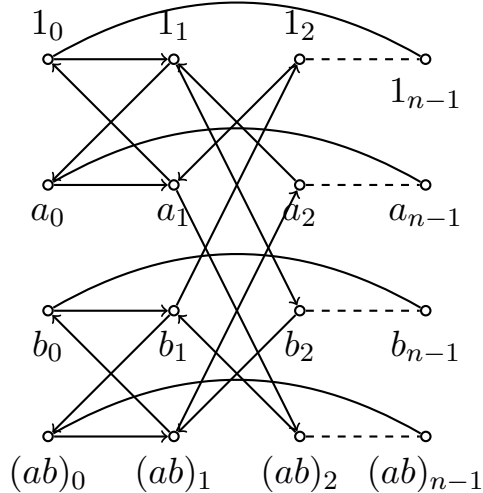
\begin{figure}[!ht]
\begin{tikzpicture}[node distance=1.2cm,thick,scale=0.7,every node/.style={transform shape},scale=1.8]
\node[circle](A0){};
\node[right=of A0, circle,draw, inner sep=1pt, label=above:{$1_0$}](D0){};
\node[right=of D0, circle,draw, inner sep=1pt, label=above:{$1_1$}](E0){};
\node[below=of D0, circle,draw, inner sep=1pt, label=below:{$a_0$}](D1){};
\node[below=of D1, circle,draw, inner sep=1pt, label=below:{$b_0$}](D2){};
\node[below=of D2, circle,draw, inner sep=1pt, label=below:{$(ab)_0$}](D3){};

\node[right=of E0, circle,draw, inner sep=1pt, label=above:{$1_2$}](F0){};
\node[below=of E0, circle,draw, inner sep=1pt, label=below:{$a_1$}](E1){};
\node[below=of E1, circle,draw, inner sep=1pt, label=below:{$b_1$}](E2){};
\node[below=of E2, circle,draw, inner sep=1pt, label=below:{$(ab)_1$}](E3){};

\node[below=of F0, circle,draw, inner sep=1pt, label=below:{$a_2$}](F1){};
\node[below=of F1, circle,draw, inner sep=1pt, label=below:{$b_2$}](F2){};
\node[below=of F2, circle,draw, inner sep=1pt, label=below:{$(ab)_2$}](F3){};

\node[right=of F0, circle,draw, inner sep=1pt, label=below:{$1_{n-1}$}](G0){};
\node[below=of G0, circle,draw, inner sep=1pt, label=below:{$a_{n-1}$}](G1){};
\node[below=of G1, circle,draw, inner sep=1pt, label=below:{$b_{n-1}$}](G2){};
\node[below=of G2, circle,draw, inner sep=1pt, label=below:{$(ab)_{n-1}$}](G3){};

\draw[->] (D0) to (E0);
\draw[->] (D1) to (E1);
\draw[->] (D2) to (E2);
\draw[->] (D3) to (E3);

\draw[->] (E0) to (D1);
\draw[->] (E1) to (D0);
\draw[->] (E2) to (D3);
\draw[->] (E3) to (D2);

\draw[->] (E0) to (F2);
\draw[->] (E1) to (F3);
\draw[->] (E2) to (F0);
\draw[->] (E3) to (F1);

\draw[->] (F0) to (E1);
\draw[->] (F1) to (E0);
\draw[->] (F2) to (E3);
\draw[->] (F3) to (E2);

\draw[dashed] (F0) to (G0);
\draw[dashed] (F1) to (G1);
\draw[dashed] (F2) to (G2);
\draw[dashed] (F3) to (G3);
\draw (D0) to [bend left] node [above]{} (G0);
\draw (D1) to [bend left] node [above]{}(G1);
\draw (D2) to [bend left] node [above]{}(G2);
\draw (D3) to [bend left] node [above]{}(G3);
\end{tikzpicture}
\caption{$n$-$\rm PDR$s for $\mathbb{Z}_2^2$ with $n\geq 3$}\label{Fig2}
\end{figure}

By Figure~\ref{Fig2}, for every $g_1\in G_1$, $g_1$  is incident to no undirected edge of $\Gamma$; however, for $i\in\mz_n\setminus\{1\}$,  for every $g_i\in G_i$, $g_i$ is incident to at least one undirected edge of $\Gamma$. Thus, $A$ fixes $G_1$ setwise. Furthermore,
\begin{itemize}
\item $\G[G_i\cup G_{i+1}]$ is a perfect matching, for every $i\in\mz_n\setminus\{0,1\}$,
\item  $\G[G_0 \cup G_1]$ and $\G[G_1\cup G_2]$ are both  union of two directed cycles of length $4$,  and
\item in all other cases, $\G[G_j\cup G_k]$ is the empty graph.
\end{itemize}
Thus, $A$ fixes $G_1$ and $G_0\cup G_2$ setwise.

Suppose $A$ does not fix setwise $G_0$. Then, there exists $\a\in A$ with $1_0^{\a}=1_2$. Observe that  $1_0,1_{n-1},1_{n-2},\ldots,1_2$ is the unique undirected
path of length $n-2$ passing through $1_0$, and also the unique undirected
path of length $n-2$ passing through $1_2$. Then $1_0^{\a}=1_2$ implies $1_2^{\a}=1_0$, that is, $\a$ interchanges $1_0$ and $1_2$. Using Figure~\ref{Fig2}, it is readily seen that
$\a$ interchanges the two directed cycles $(1_0,1_1,a_0,a_1)$ and $(1_2,a_1,(ab)_2,b_1)$, that is,
\begin{align*}
~~~~(1_2,a_1,(ab)_2,b_1)&=(1_0,1_1,a_0,a_1)^\a=
(1_0^\a,1_1^\a,a_0^\a,a_1^\a)=
(1_2,1_1^\a,a_0^\a,a_1^\a),
\\ (1_0,1_1,a_0,a_1)&=(1_2,a_1,(ab)_2,b_1)^\a=
(1_2^\a,a_1^\a,(ab)_2^\a,b_1^\a)=
(1_0,a_1^\a,(ab)_2^\a,b_1^\a).
\end{align*}
Thus $1_1=a_1^\a=b_1,$ which is a contradiction. Thus, $A$ fixes $G_0$ and $G_2$ setwise. Since $G$ is transitive on $G_0$, from the Frattini argument, we have $A=GA_{1_0}$. Moreover, now that we know that $A_{1_0}$ fixes setwise $G_0$, $G_1$ and $G_2$, it is easy to see that $A_{1_0}$ fixes $G_0$, $G_1$ and $G_2$ pointwise, because $A_{1_0}$ fixes all directed cycles in $\G[G_0\cup G_1]\cup \G[G_1\cup G_2]$. Since $\G[G_2\cup G_3]$ is a perfect matching, $A_{1_0}$ fixes $G_3$ pointwise, and arguing inductively, $A_{1_0}$ fixes $G_i$ pointwise for every $i\in\mz_n$. It follows that $A=GA_{1_0}=G$, that is, $\Gamma$ is an $n$-PDR.

\smallskip

We now consider the case $G\cong\mathbb{Z}_2^3$. Let $a,b,c\in G$ with $G=\langle a\rangle\times \langle b\rangle\times \langle c\rangle$ and let
\begin{eqnarray*}
T_{i,i+1}&=&T_{i+1,i}:=\{1\},\,\,\mbox{ when } i\in\mz_n\setminus\{0,1\}, \\
T_{0,1}&:=&\{a\},\,\,T_{1,0}:=\{b\},\,\,T_{1,2}:=\{a\},\,\,T_{2,1}:=\{c\},\\
T_{j,k}&:=&\emptyset,\,\,\hbox{ otherwise}.
\end{eqnarray*}
Let $\Gamma=\Cay(G,T_{i,j}: i,j\in\mz_n)$ and let $A=\Aut(\G)$. As usual, we have drawn $\Gamma$ in Figure~\ref{Fig3}. Then $\Gamma$ is a regular $n$-partite Cayley digraph of $G$.

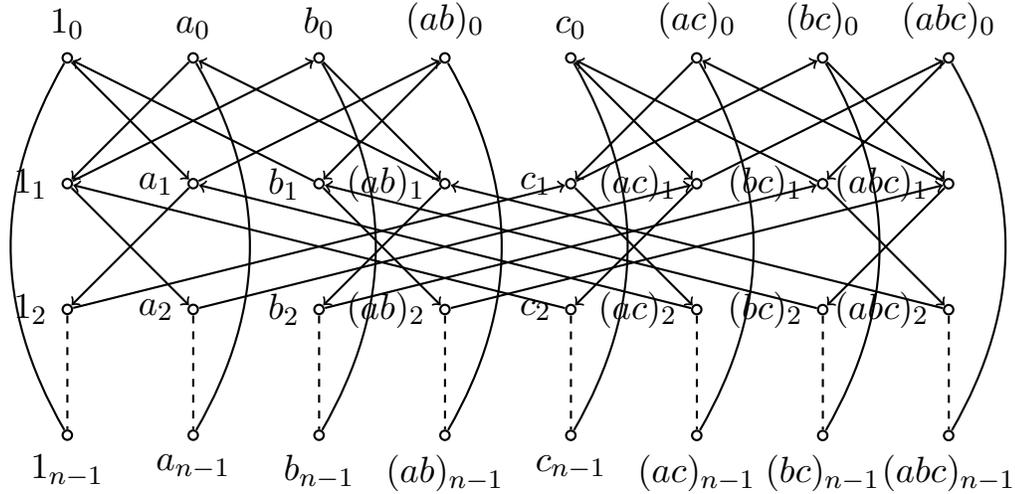
\begin{figure}[!ht]
\begin{tikzpicture}[node distance=1.2cm,thick,scale=0.7,every node/.style={transform shape},scale=1.8]
\node[circle](A0){};
\node[right=of A0, circle,draw, inner sep=1pt, label=above:{$1_0$}](D0){};
\node[below=of D0, circle,draw, inner sep=1pt, label=left:{$1_1$}](E0){};
\node[right=of D0, circle,draw, inner sep=1pt, label=above:{$a_0$}](D1){};
\node[right=of D1, circle,draw, inner sep=1pt, label=above:{$b_0$}](D2){};
\node[right=of D2, circle,draw, inner sep=1pt, label=above:{$(ab)_0$}](D3){};
\node[right=of D3, circle,draw, inner sep=1pt, label=above:{$c_0$}](D4){};
\node[right=of D4, circle,draw, inner sep=1pt, label=above:{$(ac)_0$}](D5){};
\node[right=of D5, circle,draw, inner sep=1pt, label=above:{$(bc)_0$}](D6){};
\node[right=of D6, circle,draw, inner sep=1pt, label=above:{$(abc)_0$}](D7){};

\node[below=of E0, circle,draw, inner sep=1pt, label=left:{$1_2$}](F0){};
\node[right=of E0, circle,draw, inner sep=1pt, label=left:{$a_1$}](E1){};
\node[right=of E1, circle,draw, inner sep=1pt, label=left:{$b_1$}](E2){};
\node[right=of E2, circle,draw, inner sep=1pt, label=left:{$(ab)_1$}](E3){};
\node[right=of E3, circle,draw, inner sep=1pt, label=left:{$c_1$}](E4){};
\node[right=of E4, circle,draw, inner sep=1pt, label=left:{$(ac)_1$}](E5){};
\node[right=of E5, circle,draw, inner sep=1pt, label=left:{$(bc)_1$}](E6){};
\node[right=of E6, circle,draw, inner sep=1pt, label=left:{$(abc)_1$}](E7){};

\node[right=of F0, circle,draw, inner sep=1pt, label=left:{$a_2$}](F1){};
\node[right=of F1, circle,draw, inner sep=1pt, label=left:{$b_2$}](F2){};
\node[right=of F2, circle,draw, inner sep=1pt, label=left:{$(ab)_2$}](F3){};
\node[right=of F3, circle,draw, inner sep=1pt, label=left:{$c_2$}](F4){};
\node[right=of F4, circle,draw, inner sep=1pt, label=left:{$(ac)_2$}](F5){};
\node[right=of F5, circle,draw, inner sep=1pt, label=left:{$(bc)_2$}](F6){};
\node[right=of F6, circle,draw, inner sep=1pt, label=left:{$(abc)_2$}](F7){};

\node[below=of F0, circle,draw, inner sep=1pt, label=below:{$1_{n-1}$}](G0){};
\node[right=of G0, circle,draw, inner sep=1pt, label=below:{$a_{n-1}$}](G1){};
\node[right=of G1, circle,draw, inner sep=1pt, label=below:{$b_{n-1}$}](G2){};
\node[right=of G2, circle,draw, inner sep=1pt, label=below:{$(ab)_{n-1}$}](G3){};
\node[right=of G3, circle,draw, inner sep=1pt, label=below:{$c_{n-1}$}](G4){};
\node[right=of G4, circle,draw, inner sep=1pt, label=below:{$(ac)_{n-1}$}](G5){};
\node[right=of G5, circle,draw, inner sep=1pt, label=below:{$(bc)_{n-1}$}](G6){};
\node[right=of G6, circle,draw, inner sep=1pt, label=below:{$(abc)_{n-1}$}](G7){};

\draw[->] (D0) to (E1);
\draw[->] (D1) to (E0);
\draw[->] (D2) to (E3);
\draw[->] (D3) to (E2);
\draw[->] (D4) to (E5);
\draw[->] (D5) to (E4);
\draw[->] (D6) to (E7);
\draw[->] (D7) to (E6);

\draw[->] (E0) to (F1);
\draw[->] (E1) to (F0);
\draw[->] (E2) to (F3);
\draw[->] (E3) to (F2);
\draw[->] (E4) to (F5);
\draw[->] (E5) to (F4);
\draw[->] (E6) to (F7);
\draw[->] (E7) to (F6);

\draw[->] (E0) to (D2);
\draw[->] (E1) to (D3);
\draw[->] (E2) to (D0);
\draw[->] (E3) to (D1);
\draw[->] (E4) to (D6);
\draw[->] (E5) to (D7);
\draw[->] (E6) to (D4);
\draw[->] (E7) to (D5);

\draw[->] (F0) to (E4);
\draw[->] (F1) to (E5);
\draw[->] (F2) to (E6);
\draw[->] (F3) to (E7);
\draw[->] (F4) to (E0);
\draw[->] (F5) to (E1);
\draw[->] (F6) to (E2);
\draw[->] (F7) to (E3);

\draw[dashed] (F0) to (G0);
\draw[dashed] (F1) to (G1);
\draw[dashed] (F2) to (G2);
\draw[dashed] (F3) to (G3);
\draw[dashed] (F4) to (G4);
\draw[dashed] (F5) to (G5);
\draw[dashed] (F6) to (G6);
\draw[dashed] (F7) to (G7);

\draw (D0) to [bend right] node [above]{} (G0);
\draw (D1) to [bend left] node [above]{}(G1);
\draw (D2) to [bend left] node [above]{}(G2);
\draw (D3) to [bend left] node [above]{}(G3);
\draw (D4) to [bend left] node [above]{}(G4);
\draw (D5) to [bend left] node [above]{}(G5);
\draw (D6) to [bend left] node [above]{}(G6);
\draw (D7) to [bend left] node [above]{}(G7);
\end{tikzpicture}
\caption{$n$-$\rm PDR$s for $\mathbb{Z}_2^3$ with $n\geq 3$}\label{Fig3}
\end{figure}

It is readily seen that
\begin{itemize}
\item $\G[G_0 \cup G_1]$ and $\G[G_1 \cup G_2]$ are both a union of four directed cycles of length $4$,
\item $\G[G_i \cup G_{i+1}]$ is a perfect matching for every $i\in \mz_n\setminus\{0,1\}$, and
\item in all other cases, $\G[G_j \cup G_k]$ is the empty graph.
\end{itemize}
From this, it follows that $A$ fixes $G_1$ and $G_0\cup G_2$ setwise. Note that, if $\a\in A$ interchanges $1_0$ and $1_2$, then it interchanges the directed cycles $(1_0,a_1,(ab)_0,(ab)_1)$ and $(1_2,c_1,(ac)_2,a_1)$, which is impossible. Then, arguing as in the case $G\cong\mz_2^2$, we obtain  that $A_{1_0}$ fixes  each of $G_0$, $G_1$ and $G_2$ setwise. Using Figure~\ref{Fig3}, we see that the only automorphism of $\G[G_0\cup G_1]\cup\G[G_1\cup G_2]$ fixing setwise $G_0$, $G_1$ and $G_2$ and fixing $1_0$ is the identity. Therefore $A_{1_0}$ fixes $G_0\cup G_1\cup G_2$ pointwise. From this it follows that $A_{1_0}$ fixes $G_i$ pointwise, for every $i\in \mz_n$. Thus $A=GA_{1_0}=G$ and $\Gamma$ is an $n$-PDR of $G$. \end{proof}

\section{Proof of Theorem~\ref{theo=main}}

In this section, we prove Theorem~\ref{theo=main} and, to do that, we need one more auxiliary result.

\begin{lem}\label{lem=HDR}
If the finite group $G$ admits a $2$-$\mathrm{PDR}$, then $G$ admits an $n$-$\mathrm{PDR}$ for every $n\geq 2$.
\end{lem}

\begin{proof}
Let $G$ be a finite group admitting a $2$-PDR. By Proposition~\ref{prop=HDR},
\begin{equation}\label{eq:new1}
G\ncong \mz_1,\mz_2,\mz_2^2,\mz_2^3 \hbox{ and }\mz_3.
\end{equation} In particular, $|G|\geq 4$.
We consider two cases in turn: first $G$ admits a DRR, then  $G$ admits no DRR.

\vskip 0.2cm
\noindent\textsc{\bf Case 1: }$G$ admits a DRR.
\vskip 0.2cm

By Propositions~\ref{prop=CayleySet} and~\ref{prop=HDRSet}, $G$ has subsets $R$ and $L$ such that   $\Cay(G,R)$ is a DRR and $\Cay(G,R\cup\{1\},L \cup\{1\})$ is a $2$-PDR, where $1\notin R$, $|R|<(|G|-1)/2$, $L\subseteq G\setminus (R^{-1}\cup\{1\})$ and $|L|=|R|$. We now divide the proof in two further cases, depending on whether $G$ is an elementary abelian $2$-group or not.

\vskip 0.2cm
\noindent\textsc{\bf Case 1.1: }$G$ is an elementary abelian $2$-group.
\vskip 0.2cm

Since $\Cay(G,R)$ is a DRR and $|G|>2$, $\Cay(G,R)$ is connected and hence $G=\langle R\rangle$. In particular, $|R|\geq 2$. From~\eqref{eq:new1}, we have $G\cong \mz_2^k$ with $k\geq 4$ and, since $|R|<(|G|-1)/2$, $G$ has a non-identity element $b\in G\setminus R$. Let $\G$ be the $n$-Cayley digraph $\Cay(G,T_{i,j}: i,j\in\mz_n)$, where
\begin{eqnarray}\label{eq1}
T_{i,i+1}&:=&R\cup\{1\} \mbox{ for } i\in\mz_n;\ \  T_{1,0}:=L\cup\{1\},\ T_{i+1,i}:=R\cup\{b\} \mbox{ for } i\in\mz_n\setminus\{0\};\\
\label{eq2}
T_{j,k}&:=&\emptyset\ \mbox{ for }\ j,k\in\mz_n \mbox{ with } j\not=k\pm 1.
\end{eqnarray}

It is readily seen that $\G$ is a regular $n$-partite Cayley digraph. To prove that $\G$ is an $n$-PDR, it suffices to show that $A:=\Aut(\G)=G$. Note that $R=R^{-1}$ and $L=L^{-1}$,  because $G$ is an elementary abelian $2$-group.

Let $i,j\in\mz_n$. Observe that,
\begin{itemize}
\item when $i\not=j\pm 1$, $\G[G_i \cup G_j]$ is the empty graph (this follows from~\eqref{eq2}),
\item every vertex in $\Gamma[G_0\cup G_1]$ is adjacent to only one undirected edge (this follows from~\eqref{eq1}, because $T_{0,1}^{-1}\cap T_{1,0}=\{1\}$),
\item when $i\not=0$, every vertex
in $\Gamma[G_i\cup G_{i+1}]$ is adjacent to $|R|$ undirected edges (this follows from~\eqref{eq1},
because $T_{i,i+1}^{-1}\cap T_{i+1,i}=R$).
\end{itemize}
It follows that every vertex in $G_0$ and $G_1$ is incident to
$|R|+1$ undirected edges and, when $i\in \mz_n\setminus\{0,1\}$, every vertex in $G_i$  is incident to $2|R|$ undirected edges. Since $|R|\geq 2$, we have $|R|+1\ne 2|R|$ and hence $A$ fixes $G_0\cup G_1$ setwise. Therefore $A$ induces a group of automorphisms on $\G[G_0\cup G_1]$.

Note that $$\G[G_0\cup G_1]=\Cay(G,T_{0,1},T_{1,0})=\Cay(G,R\cup\{1\},L\cup\{1\}).$$ Since  $\Cay(G,R\cup\{1\},L \cup\{1\})$ is a $2$-PDR, the automorphism group of $\G[G_0\cup G_1]$ is $G$ and hence $A$ fixes $G_0$ and $G_1$ setwise and $A_{1_0}$ fixes $G_0$ and $G_1$ pointwise.

From~\eqref{eq1},  $T_{1,2}^{-1}\cap T_{2,1}=R$ and $T_{1,2}=R\cup\{1\}$. It follows that,
 for every $g_1\in G_1$, there exists a unique $h_2\in G_2$ such that $(g_1,h_2)$ is an arc of $\G[G_1\cup G_{2}]$ and $(h_1,g_2)$ is not an arc of $\G$. From this observation and from the fact that $A_{1_0}$ fixes $G_1$ pointwise, we have that $A_{1_0}$ fixes $G_2$ pointwise. Now, when $3\leq i\leq n-1$, as $T_{i,i+1}=R\cup\{1\}$ and $T_{i+1,i}=R\cup \{b\}$, with an entirely similar argument and with an elementary induction, we get that $A_{1_0}$ fixes $G_i$ pointwise. Therefore, $A_{1_0}=1$ and $A=GA_{1_0}=G$.

\vskip 0.2cm
\noindent\textsc{\bf Case 1.2: }$G$ is not an elementary abelian $2$-group.
\vskip 0.2cm

Let $a$ be an arbitrary element of $G$ having order $o$ at least $2$. Recall that $1\notin R$, $|R|<(|G|-1)/2$, $L\subseteq G\setminus (R^{-1}\cup\{1\})$ and $|L|=|R|$. Thus, $|R|+|L|<|G|-1$. From this we deduce that $G$ has three subsets $S$, $W$ and $K$ such that
\begin{itemize}
\item $|S|=|R|+1$ with  $1,a\in S$;
\item $W\subseteq G\setminus S^{-1}$ with $|W|=|R|-1$,
\item $K:=W\cup\{1,a^{-1}\}.$
\end{itemize}
Observe that $|R|+1=|L|+1=|S|=|K|$ and $S^{-1}\cap K=\{1,a^{-1}\}$. Set
$$T_{0,1}:=R\cup\{1\}, T_{1,0}:=L\cup\{1\};\ T_{i,i+1}:=S,\ T_{i+1,i}:=K, T_{j,k}:=\emptyset,\ i,j,k\in\mz_n \mbox{ with } i\not=0, \ j\not=k\pm 1.$$

Let $\G:=\Cay(G,T_{i,j}:i,j\in \mz_n)$ and let $A:=\Aut(\G)$. Then $\G$ is a regular $n$-partite Cayley digraph of $G$. Furthermore,
\begin{itemize}
\item when $i\not=j\pm 1$, $\Gamma[G_i\cup G_j]$ is the empty graph,
\item when $i\not=0$, every vertex in $\Gamma[G_i\cup G_{i+1}]$ has two undirected edge (because
$T_{i,i+1}\cap T_{i+1,i}^{-1}=S\cap K^{-1}=\{1,a\}$),
\item every vertex in $\Gamma[G_0\cup G_1]$ has only one undirected edge (because $T_{0,1}\cap T_{1,0}^{-1}=\{1\}$).
\end{itemize}
It follows that every vertex in $G_0$ and $G_1$ is incident to $3$ undirected edges and, when $i\in\mz_n\setminus\{0,1\}$, every vertex in $G_i$ is incident to $4$ undirected edges. Thus, $A$ fixes $G_0\cup G_1$ setwise. Since  $\G[G_0\cup G_1]=\Cay(G,R\cup\{1\},L\cup\{1\})$ is a $2$-PDR, $A$ fixes $G_0$ and $G_1$ setwise and  $A_{1_0}$ fixes $G_0$ and $G_1$ pointwise.

Since $T_{1,2}\cap T_{2,1}^{-1}=\{1,a\}$, every vertex in $G_1$ has two undirected edges with the other ends in $G_2$ and, similarly, every vertex in $G_2$ has two undirected edges with the other ends in $G_1$.  Therefore, all edges in $\G[G_1\cup G_2]$ consist of cycles of length $2o$. Since $o\geq 3$ and $A_{1_0}$ fixes $G_1$ pointwise, $A_{1_0}$ fixes $G_2$ pointwise. Similarly, since $T_{i,i+1}\cap T_{i+1,i}^{-1}=\{1,a\}$ for every $i\in\mz_n\setminus\{0\}$, we have that $A_{1_0}$ fixes $G_i$ pointwise for every $i\in \mz_n$. It follows that $A_{1_0}=1$ and $A=GA_{1_0}=G$, that is, $\G$ is an $n$-PDR.

\vskip 0.2cm
\noindent\textsc{\bf Case 2: } $G$ admits no DRR.
\vskip 0.2cm
Recall that $G\ncong \mz_1,\mz_2,\mz_2^2,\mz_2^3$ and $\mz_3$. By Proposition~\ref{prop=DRR}, $G\cong \mz_2^4$, $Q_8$, or $\mz_3^2$. First we claim that $G$ has three subsets $R$, $L$ and $K$ such that
\begin{itemize}
\item  $\Cay(G,R,L)$ is a $2$-PDR,
\item $|R\cap L^{-1}|=1$ and $|R|=|L|\geq 3$,
\item $|K|=|R|=|L|$ and $|R\cap K^{-1}|=|R|-1$.
\end{itemize}
Indeed, when $G=\langle a\rangle\times\langle b\rangle\times\langle c\rangle\times \langle d\rangle\cong \mz_2^4$, take
$$R:=\{1,a,b,c,d,ad\},\, L:=\{1,ac,bc,abc,abd,bcd\}\hbox{ and }K:=\{ab\}\cup(R^{-1}\setminus\{1\}),$$
when $G=\langle a,b~|~a^4=b^4=1,b^2=a^2,a^b=a^{-1}\rangle\cong Q_8$, take
$$R:=\{1,a,b\},\,L:=\{a^2,b^{-1},ab\}\hbox{ and }K:=\{ab\}\cup(R^{-1}\setminus\{1\}),$$
when $G=\langle a\rangle\times\langle b\rangle\cong \mz_3^2$, take
$$R:=\{1,a,b\},\,L:=\{a,b^{-1},ab\}\hbox{ and }K:=\{ab\}\cup(R^{-1}\setminus\{1\}).$$
Now, an elementary computation with the computer algebra system \textsc{Magma}~\cite{magma} reveals that with these choices of $R$, $L$ and $K$ all three conditions above are satisfied.


To finish the proof, it suffices to show that $G$ admits an $n$-PDR for every $n\geq 3$. Set
$$T_{0,1}:=R,\ T_{1,0}:=L;\ T_{i,i+1}:=R,\ T_{i+1,i}:=K, \ T_{j,k}:=\emptyset,\ i,j,k\in\mz_n \mbox{ with } i\not=0,\ j\not=k\pm 1.$$
Let $\G=\Cay(G,T_{i,j}:i,j\in \mz_n)$. Then every vertex in $\Gamma[G_0\cup G_1]$ is incident to exactly one undirected edge and, for $i\not=0$, every vertex
in $\Gamma[G_i\cup G_{i+1}]$ is incident to exactly $|R|$ undirected edges and one directed edge. This yields that the automorphism group of $\Gamma$ fixes setwise $G_0\cup G_1$. Now, to conclude the proof we argue as in Case~1.1 and we obtain that $\G$ is an
$n$-PDR. \end{proof}

\begin{proof}[Proof of Theorem~$\ref{theo=main}$] From~\eqref{boundn}, we have $n\ge 3$. To finish the proof, it suffices to show that $G$ admits  no $n$-PDR if and only if $3\leq n\leq 5$ and $G=\mz_1$.

The sufficiency follows from Lemma~\ref{lem=small}. To prove the necessity, by Lemma~\ref{lem=HDR}, we may assume that $G$ admits no $2$-PDR. Now, by Proposition~\ref{prop=HDR}, $G\cong \mz_1,\mz_2,\mz_2^2,\mz_2^3$ or $\mz_3$. By Lemmas~\ref{lem=small} and \ref{lem=2^2,2^3}, $G\cong\mz_1$ and $3\leq n\leq 5$.
\end{proof}

\f {\bf Acknowledgement:} This work was supported by the National Natural Science Foundation of China (11731002), the 111 Project of China (B16002), the Natural Science Foundation of Jiangsu Province, China (BK20200627) and by the China Postdoctoral Science Foundation (2021M693423).


\begin{thebibliography}{99}

\bibitem{Babai}
L. Babai, Finite digraphs with given regular automorphism groups,
\textit{Period. Math. Hungar.} \textbf{11} (1980), 257--270.

\bibitem{BabaiI}
L. Babai, W. Imrich, Tournaments with given regular group,
\textit{Aequationes Math.} \textbf{19} (1979), 232--244.

\bibitem{magma}
W. Bosma, C. Cannon, C. Playoust, The MAGMA algebra system I: The user language,
\textit{J. Symbolic Comput.} \textbf{24} (1997), 235--265.


\bibitem{G1}M.~Grech, Graphical cyclic permutation groups, \textit{Discrete Math.} \textbf{337} (2014), 25--33.

\bibitem{G2}M.~Grech, A.~Kisielewicz, Cyclic permutation groups that are automorphism groups of graphs, \textit{Graphs Combin.} \textbf{35} (2019), 1405--1432.


\bibitem{G3}M.~Grech, A.~Kisielewicz, Graphical representations of cyclic permutation groups, \textit{Discrete Appl. Math.} \textbf{277} (2020), 172--179.


\bibitem{DSV}
E.~Dobson, P.~Spiga, G.~Verret, Cayley graphs on abelian groups, \textit{Combinatorica} \textbf{36} (2016), 371--393.

\bibitem{DFS}
J.~-L.~Du, Y.~-Q.~Feng, P.~Spiga, A classification of the graphical $m$-semiregular representations of finite groups, \textit{J. Combin. Theory Ser. A} \textbf{171} (2020), 105174.

\bibitem{DFS1}
J.~-L.~Du, Y.~-Q.~Feng, P.~Spiga, On the existence and the enumeration of bipartite regular representations of Cayley graphs over abelian groups, \textit{J. Graph Theory} \textbf{95} (2020), 677--701.
\bibitem{DFS3}
J.~-L.~Du, Y.~-Q.~Feng, P.~Spiga, A conjecture on bipartite graphical regular representations,\textit{Discrete Math.} \textbf{343} (2020), 111913.

\bibitem{DFS2}
J.~-L.~Du, Y.~-Q.~Feng, P.~Spiga, On Haar digraphical representations of groups, \textit{Discrete Math.} \textbf{343} (2020), 112032.

%


\bibitem{MSV}J.~Morris, P.~Spiga, G.~Verret, Automorphisms of Cayley graphs on generalised dicyclic groups, \textit{European J. Combin.} \textbf{43} (2015), 68--81.

\bibitem{MorrisSpiga1}
J. Morris, P. Spiga, Every finite non-solvable group admits an oriented regular representation, \textit{J. Combin. Theory Ser. B} \textbf{126} (2017), 198--234.

\bibitem{MorrisSpiga3}J. Morris, P.~Spiga, Classification of finite groups that admit an oriented regular representation, \textit{Bull. Lond. Math. Soc.} \textbf{50} (2018), 811--831.

\bibitem{Spiga}
P. Spiga, Finite groups admitting an oriented regular representation,
\textit{J. Combin. Theory Ser. A} \textbf{153} (2018), 76--97.
\bibitem{Spiga2}
P. Spiga, On the Existence of Frobenius Digraphical Representations, \textit{Electron. J. Combin.} \textbf{25} (2018), $\sharp$P2.6.

\bibitem{Sabidussi}
G. Sabidussi, On a class of fix-point-free graphs, Proc. Amer. Math. Soc. \textbf{9} (1958), 800-804.

\bibitem{Xiaf}
B.~Z.~Xia, T.~Fang, Cubic graphical regular representations of $\mathrm{PSL}_2(q)$, \textit{Discrete Math.} \textbf{339} (2016), 2051--2055.

\end{thebibliography}
\end{document}